\documentclass[11pt]{amsart}
\usepackage{geometry}
\usepackage{amssymb,amsmath,amsthm}
\usepackage{geometry}
\theoremstyle{plain}
\usepackage{hyperref}
\usepackage{comment}
\theoremstyle{plain}

\newtheorem{theorem}{Theorem}

\newtheorem{lem}{Lemma}[section]

\newtheorem{lemma}[lem]{Lemma}

\newtheorem{proposition}[lem]{Proposition}

\newtheorem{thm-n}[lem]{Theorem}

\theoremstyle{remark}
\newtheorem{rem}[lem]{Remark}

\newtheorem*{rem*}{Remark}
\newtheorem*{notat*}{Notation}
\newtheorem*{exm*}{Example}
\theoremstyle{definition}

\newcommand{\ind}{{\rm ind}}

\newcommand{\Mon}{{\rm Mon}}
\newcommand{\soc}{{\rm soc}}
\newcommand{\orb}{{\rm Orb}}
\newcommand{\U}{{\rm OR}}
\newcommand{\M}{{\rm mfpr}}
\newcommand{\mC}{{\mathbb C}}
\newcommand{\Aut}{{\rm Aut}}
\newcommand{\PSL}{{\rm PSL}}
\newcommand{\PSU}{{\rm PSU}}
\newcommand{\PSp}{{\rm PSp}}
\newcommand{\PO}{{\rm P\Omega}}

\usepackage{color}
\newcommand{\EILIDH}[1]{\textcolor{red}{#1}}
\newcommand{\SPENCER}[1]{\textcolor{blue}{#1}}

\begin{document}

\title[Low-genus primitive monodromy groups of type B]{Low-genus primitive monodromy groups with a nonunique minimal normal subgroup}
%Genus one monodromy groups of type B}
\def\technion{Department of Mathematics, Technion - Israel Institute of Technology, Haifa, Israel}
\author{Spencer Gerhardt}
\address{Department of Mathematics, University of Southern California, Los Angeles, CA, USA.}
\email{sgerhard@usc.edu}
\author{Eilidh McKemmie}
\address{Department of Mathematical Sciences, Kean University, Union, NJ, USA.}
\email{emckemmi@kean.edu}
\author{Danny Neftin}
\address{\technion}
\email{dneftin@technion.ac.il}%

\begin{abstract}
    Let $X$ be a Riemann surface, and let $f:X\to\mathbb P^1_\mC$ be an indecomposable (branched) covering of genus $g$ and degree $n$ whose monodromy group has more than one minimal normal subgroup. 
    Closing a gap in the literature, we show that there is only one such covering when $g\leq 1$. Moreover, for arbitrary $g$, there are no such coverings with $n\gg_g 0$ sufficiently large.  
\end{abstract}

\maketitle
%\thanks{This work was partially supported by the AIM SQuaREs program.}

\section{Introduction}

Fix an integer $g\geq 0$, and consider  degree-$n$ (branched) coverings $f:X\to\mathbb P^1_{\mathbb C}$ of the Riemann sphere $\mathbb P^1_{\mathbb C}$ by a (connected compact) Riemann\footnote{Alternatively, throughout, $f$ can be picked to be a morphism from a smooth projective algebraic curve $X$.} surface $X$ of genus $g_X=g$.  The classification of monodromy groups aims to determine those coverings $f$ whose monodromy group $G=\Mon_\mC(f)\leq S_n$ is 
not (the generically occurring) $A_n$ or $S_n$ when $n\gg_g 0$ is sufficiently large. %, the ones occuring generically. 
This classification has far-reaching implications throughout mathematics, some of which are discussed in \cite[\S 1, pg.\ 3]{NZ}. As covers of genus $g=0$ or $1$ play a key role in such implications,  the classification in such genera, known as the genus-$0$ program, furthermore seeks to determine the coverings {\it in all degrees $n$} with monodromy group $\neq A_n,S_n$. % for $g=0$ and $1$

For decomposable maps,  $\Mon(f)$ is clearly a subgroup of the stabilizer $S_d\wr S_{n/d}$, $1<d<n$, of a nontrivial partition of $\{1,\ldots,n\}$, and hence smaller than  $A_n$ or $S_n$. Henceforth, the classfication restricts to coverings $f$ as above that are {\it indecomposable}, that is, cannot be written as $f=g\circ h$ for coverings $g,h$ of degrees $>1$. Such coverings have {\it primitive} monodromy groups $G\leq S_n$, that is, transitive groups that do not preserve any nontrivial partition of $\{1,\ldots,n\}$.  

\begin{comment} Since we are assuming $X$ is connected, $G$ must be a transitive permutation group \cite[Section~1]{GuralnickThompson_1990}. In addition, $G$ will be primitive if the cover $\psi$ is primitive, i.e. if $\psi$ does not factor nontrivially through any surface. 
\end{comment}

The Aschbacher--O'Nan--Scott structure theory   divides primitive groups $G\leq S_n$ into several families A-C, see \cite[Thm.\ 11.3]{guralnick2003monodromy} or \cite{GuralnickThompson_1990}. The (primitive) groups of type B are those admitting more than one minimal normal subgroup. 
The  type-B  (primitive) monodromy groups of genus-$0$ indecomposable coverings $f$ were determined by Shih \cite{Shih_1991}, up to a few small gaps noted below. Low-degree type-B genus-$1$ coverings $f$ were computed by Salih, see 
\cite{Salih2023}. Moreover, it was generally believed that Shih's proof should extend to genus-$1$ covers and, for $n\gg_g 0$, to arbitrary genus $g$. However, so far, such a proof has not appeared in the literature.  

\begin{comment} 
As the composition factors of monodromy groups are contained in the composition factors of primitive groups, Aschbacher and Scott's classification of finite primitive groups into types $(A),(B),(C1),(C2)$, and $(C3)$ \cite{Aschbacher_1985} provides a key tool for studying the structure theory of monodromy groups. For instance, the Guralnick-Thompson conjecture \cite{GuralnickThompson_1990}, which states the collection of composition factors for all primitive genus $g_X$ systems contains only cyclic groups, alternating groups, and finitely many other simple groups, was proved by showing this property holds for all primitive genus $g_X$ systems of types $(A),(B), (C1),(C2)$ and $(C3)$
\cite{ FrohardtMagaard_01}. 
 
In a different direction, Riemann's existence theorem can also be used to classify monodromy groups of small genus. For instance,
\cite{Aschbacher_1990} classifies the primitive genus $0$ and genus $1$ systems of type $(C2)$, \cite{Neubauer_1993} describes the 
primitive genus $0$ and $1$ systems of type $(A)$, and \cite{Shih_1991} shows there are no genus 0 systems of type $(B)$. Our paper contributes this project, classifying the genus $1$ systems of type $(B)$. The main result shows there is a unique genus $1$ monodromy system of type $(B)$, with monodromy group ${\rm PSL}_2(7).2$, appearing for a $168$-sheeted cover $\psi$.  
\end{comment}

In this note we close the above gaps in the literature by extending Shih's argument,  thereby completing the classification of type-B monodromy groups. We denote by $A.B$ a group extension of $B$ by $A$. % by $C_n$ the cyclic group o.  %that is, the conjugacy classes $C_1,\ldots,C_r$ of inertia groups
\begin{theorem}\label{main}
Let $f:X\to\mathbb P^1_{\mathbb C}$ be an indecomposable  degree-$n$ covering  of genus $g_X< \max\{2,n/5000\}$ whose monodromy group $G:=\Mon_{\mathbb C}(f)$ contains more than one minimal normal subgroup. %and type-B monodromy group $G$. 
Then $n=168$ %the ramification of $f$ is  given in \ref{thm:}, 
and   $G\cong {\rm PSL}_2(7)^2.C_2$. % is a quadratic extension
%    If $G$ is of type $(B)$ and admits a genus $1$ system $(G,S,\Omega)$ then $n=168$, $S$ has type $(2,3,8)$ and $G\cong {\rm PSL}_2(7).2$.
%    \EILIDH{Do we want to change to include c, or perhaps add another theorem for possibly larger genus?} \DN{Yes, I think we can include $c$ here. }
\end{theorem}

In fact, we'll see there are two possible ramification types for such degree-$168$ coverings, see Remark \ref{ram}. The proof follows \cite{Shih_1991} closely. When   adjustments  are required, we detail the new argument. Some of these arguments involve invoking the classification of finite simple groups. 
Note that the constant $1/5000$ appearing in Theorem~\ref{main} arises in restricting the possible monodromy groups of coverings $f$ with three branch points of ramification indices $2,3$ and $7$, and can be taken to be bigger for other types of ramification. Throughout the proof we list the relevant constants to each part of the argument explicitly.
%but it is already reasonable in view of other constants provided by the classification.}

Groups of type B have two isomorphic minimal normal subgroups, each isomorphic to a power $L^t$, $t\geq 1$ of a (nonabelian) simple group $L$. Propositions \ref{prop:br}--\ref{prop:t=1}
%, together with the computer computation in Proposition \ref{prop:Spencer}, 
allow us to restrict to:
\begin{itemize}
    \item[1)] covers $f$ with three branch points $P_1,P_2,P_3$;
    \item[2)] a short explicit (finite) list of ramification indices $e_1,e_2,e_3$ for the Galois closure of $f$ over $P_1,P_2,P_3$;
    \item[3)] the case $t=1$, so that $L^2\leq G\leq \Aut(L)^2$.
\end{itemize}

These reductions use upper bounds on the ratios between the number of fixed points of a group element and the degree $n$, also known as {\it fixed point ratios}. The bounds in  \cite{Aschbacher_1990} and \cite{Shih_1991} suffice for these reductions. 
However, to treat the case $t=1$ in Propositions \ref{prop: socle} and \ref{prop:Spencer},  
we apply newer bounds from the work of Burness and Thomas \cite{Burness_2007, Burness2021}. 
In this case, upper bounds on fixed point ratios come from lower bounds on conjugacy classes in $L$, and these are provided by \cite{Burness_2007, Burness2021} and further computations that we carry out. Moreover, we replace most of Shih's computations for this case \cite[4.24-4.35]{Shih_1991}  by automated computer checks. Our code is available at \url{https://neftin.net.technion.ac.il/files/2025/01/b-code.zip} and in the ancillary files of this paper at \url{https://arxiv.org/abs/2501.15538}.
%recorded in \cite{code}. 
%of the argument reduces  such as in cases where $f$ has three branch points with ramification indices $2,3,7$, 

%replacing \cite{Shih_1991}
%It is well known that coverings of fixed genus and  monodromy group $G\leq S_n$ correspond to tuples $x_1,\ldots,x_r$ with product $1$ that generate $G$ with restrictions on the total number of orbits of $x_1,\ldots,x_r$, cf.\ \S \ref{sec:prelim}. 

{\it Acknowledgments.} D.\ N.\  is grateful for the support of the Israel Science Foundation, grant no.\  353/21. This work was partially supported by the AIM SQuaREs program. Computer calculations were carried out using GAP \cite{GAP4} and MAGMA \cite{Magma}. 

\section{Preliminaries}\label{sec:prelim}
{\it Notation.} 
Throughout the paper $G$ is a primitive group        acting on a finite set $\Omega$ of size $n$. The socle $\soc(G)$ is the group generated by its minimal normal subgroups. 
For $x\in G$, let $|x|$ denote the order of $x$, $\orb(x)$  the number of orbits of $x$ on $\Omega$, and by $f(x)$ the number of fixed points of $x$ on $\Omega$. The fixed point ratio is ${\rm fpr}(x)=f(x)/n$. For a subset $S\subseteq G$, we denote by $\orb(S)$ the sum $\sum_{s\in S}\orb(s)$. For  $x\in G$, denote by $x^G$ its conjugacy class. 
Let $\phi$ denote Euler's totient function.

\noindent {\it Monodromy.}
Given a group $G$ acting on a set $\Omega$, and a tuple $S$ of elements $x_1,\ldots,x_r\in  G$ with product $x_1x_2\ldots x_r=1$ generating $G=\langle S\rangle$, the tuple $(G,S,\Omega)$ is called a {\it (product-$1$)  system}. 
 %As in 
 By Riemann's existence theorem (RET), for every such system $(G,S,\Omega)$ and points $P_1,\ldots,P_r\in\mathbb P^1(\mathbb C)$, there exists a covering 
 $f:X\to\mathbb P^1_{\mathbb C}$ with monodromy group 
 $G$ acting on $\Omega$, with branch locus  $\{P_1,\ldots,P_r\}$, such that the branch cycle over $P_i$ is conjugate to $x_i$ in $G$, for every $i=1,\ldots,r$.
The ramification type corresponding to the system is then the multiset of conjugacy classes $C_1,\ldots,C_r$ of $x_1,\ldots,x_r$, resp.
 %induces a {\it , where . 
 
The genus $g_X$ of $X$ then satisfies the Riemann--Hurwitz formula:
$$2(g_X-1)=-2n+\sum_{i=1}^r(n-\orb(x_i)),$$
or equivalently, 
$\U(S)=\#S-2+{2(1-g_X)}/{n}$, where the orbit ratio $\U(x):={\orb(x)}/{n}$ is the ratio between the orbit length of $x$ and $n$, and $\U(S)=\sum_{i=1}^r \U(x_i)$.  Note that $\U$  is denoted by $\mathcal U$ in \cite{Shih_1991}.
 % for any $S=\{x_1,\ldots,x_r\}\subseteq G$.. 
 
 We say that a system $(G,S,\Omega$) is a {\it genus-$g$ system}  if indeed $\U(S)=\#S-2+{2(1-g)}/{n}$, so that the genus of the covering space given by RET is indeed $g$. 
        %\begin{itemize}
            %\item $x_i\neq 1$, for $1\leq i\leq r$
            %\item 
        %\end{itemize}
        Moreover we say $S$ \textit{has type }$(|x_1|, ..., |x_r|)$ where $|x_1|\le |x_2| \le \cdots \le |x_r|$.
%\begin{notat*}
%We provide a list of all our notation for the reader's reference.
    %\begin{itemize}
        %\item $G$ is a transitive subgroup of $Sym(\Omega)$, and is of type $(B)$.
        % \item 
        %\item 
        %\begin{itemize}
            %\item 
            %\item 
            %\item 
        %\end{itemize}
 %       \item 
 %       \item $
    %\end{itemize}
%\end{notat*}

\noindent {\it Primitive groups.} 
We  consider  groups $G$ of Aschbacher-Scott type B, that is, $G$ has two  nonabelian minimal normal subgroups both isomorphic to $L^t$ where $t\ge 1$ and $L$ is a nonabelian simple group \cite[Theorem~11.2(ii)]{guralnick2003monodromy}. In this case, $\soc(G) \cong L^{2t}$.
 The action of $G$ {is transitive with stabilizer  a diagonal copy of $L^t$ in $(L^t)^2$, so that each of the minimal normal subgroups $L^t$ act regularly, and }%is on a diagonal copy of $L$ in the socle, 
  $n=|L|^t$.

\noindent {\it Preliminary lemmas.} We shall use the following lemmas. For $x\in G\setminus\{1\}$, let  $\M(x)=\max\left\{\text{fpr}(x^i) \,\middle|\, 1\le i < |x|\right\}$ denote the maximal fixed-point ratio among nontrivial elements in $\langle x\rangle$. Note $\M(x)$ is denoted by $\mathcal M(x)$ in \cite{Shih_1991}.

\begin{lemma}[{\cite[3.3]{Aschbacher_1990}}]\label{lem: U formula}
    \begin{align*}
    \orb(x)&=\frac{1}{|x|}\left(\sum_{d\mid |x|}\phi\left(\frac{|x|}{d}\right) f\left(x^d\right)\right), \\
    \U(x)&\le \frac{1}{|x|}\left(1+(|x|-1-\phi(|x|))\M(x)+\phi(|x|)\frac{f(x)}{n}\right)\\
    &\le \frac{1}{|x|}(1+\M(x)(|x|-1)).
    \end{align*}
\end{lemma}
\begin{proof}
    The first equality comes from \cite[3.3]{Aschbacher_1990}. Dividing by $n$, we get
    \begin{align*}
    \U(x)&=\frac{1}{|x|}\left(\sum_{d\mid |x|}\phi\left(\frac{|x|}{d}\right) \frac{f\left(x^d\right)}{n}\right) \\
    &= \frac{1}{|x|}\left(1+\phi(|x|)\frac{f(x)}{n}+\sum_{\substack{d\mid |x|\\ d\ne 1, |x|}}\phi\left(\frac{|x|}{d}\right) \frac{f\left(x^d\right)}{n}\right)\\
    &\le \frac{1}{|x|}\left(1+\phi(|x|)\frac{f(x)}{n}+\M(x)\sum_{d\mid |x|}\phi\left(\frac{|x|}{d}\right)\right)\\
    &\le \frac{1}{|x|}\left(1+(|x|-1-\phi(|x|))\M(x)+\phi(|x|)\frac{f(x)}{n}\right)
    \end{align*}
    by the well known identity $\sum_{d\mid |x|}\phi\left({|x|/d}\right)=|x|$. Finally, the third inequality comes from noting that $f(x)/n\le \M(x)$ by definition.
\end{proof}

\begin{lemma}\label{lem: fixed point bounds} For  $x\ne 1$ one has: 
    \item[(1)] $\M(x)\leq \frac{1}{10}$, and $\U(x)\leq \frac{11}{20}$. If $L\ne A_5$, then $\U(x)\leq \frac{8}{15}$.
%    \item[(2)] $\M(x)\le \frac{1}{15}$ and $\U(x)\leq \frac{8}{15}$ for $|x|=2$, unless $L=A_5$, $t=1$, and $x$ acts on $L$ as an outer involution, in which case $\M(x)\le \frac{1}{10}$ and $\U(x)\leq \frac{11}{20}$.
    \item[(2)] $\M(x)\leq \frac{1}{20}$, and $\U(x)\leq \frac{11}{30}$ for $|x|=3$
    \item[(3)] $\M(x)\leq \frac{1}{10}$, and $\U(x)\le \frac{13}{40}$ for $|x|=4$,
    \item[(4)] $\M(x)\leq \frac{1}{12}$, and $\U(x)\leq \frac{4}{15}$ for $|x|\ge 5$.
\end{lemma}
Note that  $\frac{4}{15}<\frac{13}{40}<\frac{11}{30}<\frac{8}{15}<\frac{11}{20}$. 

\begin{proof}
    The bounds on $\M(x)$ follow directly from \cite[4.6(2)]{Shih_1991}. The bounds on $\U(x)$ in (1) and (2) follow from \cite[4.7]{Shih_1991} and the bounds on $\U(x)$ in (3) and (4) follow from Lemma~\ref{lem: U formula} along with the associated bounds on $\M(x)$.
\end{proof}

The following lemma is a classical fact, see \cite[Prop.\ 2.4]{GuralnickThompson_1990} or \cite[Prop.\ 9.5]{NZ}. 
\begin{lemma}\label{lem:closure-1}
    Let $f:X\to \mathbb P^1$ be a degree $d$ covering with monodromy group $G$ whose corresponding system is of one of the types $(d,d),(2,2,d),(2,2,2,2),(2,3,3),(2,3,4)$, 
    $(2,3,5)$, $(2,3,6)$, $(2,4,4)$, or $(3,3,3)$. Then either $G$ is solvable, or $G\cong A_5$ and the type is $(2,3,5)$. 
\end{lemma}
In fact, in the setup of the lemma the genus of the Galois closure of $f$ is at most $1$ and the maps $f$ and their monodromy groups are completely classified \cite[Prop.\ 9.5]{NZ}. See also  \cite[Lemmas 2.5.8]{Khud} for a criterion to ensure genus $>1$ in terms of $\M$.
\section{Proof of Main Theorem}
% \begin{theorem}
% Let $\psi : X \rightarrow \mathbb{P}^1$ be an indecomposable covering of genus $g_X \le 1$ with monodromy group $G = \Mon(f)$ such that $\soc(G)$ is not one of the groups ${\rm PSL}_2(7)^2$, ${\rm PSL}_2(8)^2$, $A_N^2$ for $N \le 8$. Then G has a unique minimal normal subgroup.
% \end{theorem}
%This section is devoted to the proof of our main theorem. 
%The same proof strategy used in Shih’s paper \cite{Shih_1991} is applied with a few modifications. Our approach is to rule out all but finitely many cases which can then be checked by computer.

%Following \cite[§1]{Shih_1991}, 
Throughout the proof we assume $(G,S,\Omega)$ is a system of genus $g$, where $S$ is a tuple $x_1, \ldots, x_r$ of product $1$  generating $G$, and $G$ is a primitive group of type B acting on a set $\Omega$ of size $n$. For short, we call such a configuration a \textit{type B genus-$g$ system}. Let
$L$ be the nonabelian simple group such that $\soc(G) \cong L^{2t}$ for some $t \ge 1$.
        % \item 
Denoting by $L_1, ..., L_t$ the $t$ copies of $L$ in a minimal normal subgroup of $G$, we let $\rho:G \rightarrow S_t$ be defined by $L^x_s = L_{\rho(x)(s)}$, so that $\rho(x)$ permutes the $t$ copies of $L$.

We assume that $g\leq 1+cn$ for a constant $c\leq 1/5000$. This choice will be justified in the proof of Proposition~\ref{prop: socle} in treating systems of types $(2,3,7)$. Then $\U(S) \ge \# S - 2(1+c)$ by the Riemann--Hurwitz formula. Under Shih's assumption that $g= 0$, the strict inequality $\U(S) > \#S - 2$ holds. 
We adjust Shih's argument to work even when assuming merely that the weaker inequality 
%in fact holds when %$\U(S) > \#S - 2(1+c)$ is 
%replaced by 
$\U(S) \geq \# S - 2(1+c)$ holds. We follow Shih's paper \cite{Shih_1991} closely, indicating the required modifications and for what constant $c$ the proof works at each step. 

As in \cite[(4.8)]{Shih_1991} and \cite[(2.3)-(2.4)]{Salih2023}, one first shows that, outside one exceptional type treated in Proposition \ref{prop:Spencer} (with Magma),   $\# S \le 3$. We modify this argument to additionally calculate the explicit bound on the genus under which the conclusion holds.
\begin{proposition}\label{prop:br}
 %There exists a constant $c>0$ such that 
 For every degree-$n$ type B system $(G,S,\Omega)$ of genus {$g< n/80+1$}, one has $\#S\leq 4$. Moreover, $\#S\leq 3$ unless $L=A_5$ and $S$ is of type $(2,2,2,3)$. 
 
 %   If $g_X\le 1$ then $r\le 4$. In fact, $r\le 3$ unless $L=A_5$ and $S$ has type $(2,2,2,3)$.
\end{proposition}
\begin{proof}
%Pick $c<1/80$.    
By Lemma~\ref{lem: fixed point bounds}.(1), one has $\U(x_i) \le {11}/{20}$ % < \frac{3}{5}$ 
    so that $\U(S) < \left({11}/{20}\right)\cdot\#S$. Thus, as $g<n/80+1$, the Riemann--Hurwitz formula yields: \[\#S-2\left(1+\frac{1}{80}\right)< \U(S) < \left(\frac{11}{20}\right)\cdot\#S,\] and hence $\#S \le 4$. % for $c<1/8$. %for sufficiently small $c$.

If $S$ has type $(2,2,2,2)$, then $G$ is solvable 
by Lemma \ref{lem:closure-1}, %\cite[3.5]{Aschbacher_1990}. Therefore every minimal normal subgroup of $G$ is elementary abelian, a 
contradicting that $L$ is a nonabelian simple group. % minimal normal subgroup. 
So we may assume $S$ contains an element of order at least $3$. If $L\ne A_5$, then Lemma~\ref{lem: fixed point bounds}.(2\,-\,5) implies  $\U(S)\le 3\cdot {8}/{15}+{11}/{30}=59/30<2-2c$ for $c<1/60$, and hence $\#S\le 3$.

Finally, assume $L=A_5$. If $S$ does not have type $(2,2,2,3)$, then by Lemma~\ref{lem: fixed point bounds}.(2-3), one has $\U(S)\le 2\cdot{11}/{20}+2\cdot{11}/{30}<2-2c$ for $c<1/12$ in case two branch points are of type $>2$, or 
$\U(S)\leq 3\cdot {11}/{20}+{13}/{40}=79/40<2-2c$ for $c<1/80$ in case one branch point is of type $>3$, % sufficiently small $c$, 
a contradiction.
\end{proof}

Now consider types $(k,\ell,m)$ of length $3$, 
%the systems $x_1,x_2,x_3$ of length $3$. Let $k, \ell ,m$ be the orders of $x_1, x_2, x_3$, where 
and assume without loss of generality $k \le \ell \le m$.
\begin{proposition}\label{prop: r=3 cases}
%    There exists a constant $c>0$ such that 
    Every type B  system $(G,S,\Omega)$ of genus {$g< 1+n/296$} and $\#S=3$   has  one of the following types:
    \iffalse OLD VERSION
\begin{itemize}
    \item $(2, 3,m)$ for $m\geq 7$, with additionally $m\le 18$ if $L\ne A_5$.
    \item $(2, 4,m)$, $5 \le m \le 35$
    \item $(2, 5,m)$, $5 \le m \le 10$
    \item $(2, 6,m)$, $m = 6, 7, 8$
    \item $(3, 3,m)$, $m = 4, 5$
    \item $(3, 4, 4)$.
\end{itemize}
\fi 
\begin{itemize}
    \item $(2, 3,m)$ for $m\geq 7$, with additionally $m\le 29$ in case $L\ne A_5$.
    \item $(2, 4,m)$, $5 \le m \le 37$. 
    \item $(2, 5,m)$, $5 \le m \le 13$.
    \item $(2, 6,m)$, $m = 6, 7, 8, 9$.
    \item $(3, 3,m)$, $4\le m \le 9$.
    \item $(3, 4, m)$, $m=4, 5$.
\end{itemize}
\end{proposition}
\begin{proof}
%Pick $c<1/200$,   
Shih \cite[(4.11)]{Shih_1991} restricts the possible types $(k,\ell,m)$ of $3$-tuples using his estimates of $\U(S)$ from \cite[(4.10)]{Shih_1991}. For $g<1+n/296$, \cite[(4.10)]{Shih_1991} takes the following form (via the same argument): 
    
    (4.10') Assume $\M(x)< \lambda$ for all $x\in G$ of prime order, and $g<1+cn$ for $c>0$ sufficiently small to make the following denominators positive. Then:
    \begin{enumerate}
        \item $1>\frac{1}{k}+\frac{1}{\ell}+\frac{1}{m} > \frac{1-2c-3\lambda}{1-\lambda}.$
        \item $k\leq \left\lfloor \frac{3(1-\lambda)}{1-2c-3\lambda}\right\rfloor$.
        \item If $\M(x_1)\leq a\leq \lambda$, then $\ell\leq \left\lfloor\frac{2(1-\lambda)}{(1-1/k)(1-a)-2c-2\lambda}\right\rfloor$.
        \item If $\M(x_1)\leq a\leq \lambda$ and $\M(x_2)\leq b\leq \lambda$, then
        $$ m\leq \left[ \frac{1-\lambda}{(1-1/k)(1-a)+(1-1/\ell)(1-b)-2c-(1+\lambda)}\right].$$
    \end{enumerate}
    Shih's proof of (4.11) does not give the details of the computation and when following his method, we get larger bounds. Hence, we detail  the argument here: 
    
    First, by (4.10').(2) above with $\lambda =1/10$ and $c<1/80$, one has $k\in\{2,3\}$. For $k=2$, (4.10').(3) with $c<1/80$ yields $\ell\leq 7$. Moreover, for $\ell=7$, one may apply (4.10').(4) with $a=\lambda=1/10$, and $b=1/24$ by \cite[(4.6).(2)]{Shih_1991} to get that $m\leq 6$ for $c<1$, contradicting $\ell\leq m$. Thus $\ell\leq 6$ if $k=2$. For $k=3$, (4.10').(3) gives $\ell\leq 4$ for $c<1/50$. 
    Similarly, for $(3,3,m)$, we get $m\leq 9$ for $c<1/200$.
    For $(3,4,m)$, we get $m\leq 5$ for $c<11/160$.
    For $(2,4,m)$, we get $m\leq 37$ for $c<1/296$.
    For $(2,5,m)$, we get $m\leq 9$ for $c<1/200$.
    For $(2,6,m)$, we get $m\leq 9$ for $c<1/200$.
    If $L\neq A_5$, for $(2,3,m)$ we get $m\leq 29$ when $c<1/116$.%\DN{Why can $b=1/20$ be picked here, can do $1/40$? } 

  %  His proof applies after replacing the inequality  $\U(S) > 1$ (resp.\ $g_X<1$) by $\U(S) \ge 1-2c$ (resp. $g_X \le 1+cn$).
\end{proof}

Recall that if $G$ is a primitive group of type B, then soc($G$) $\cong L^{2t}$ for some nonabelian simple group $L$, and $t \geq 1$.
For a genus-$0$ system,  \cite[(4.17)-(4.21)]{Shih_1991} asserts that  $t = 1$, so that $\soc(G)\cong L^2$. The treatment relies on \cite[(4.16)]{Shih_1991} which applies the inequality $\U(S) \ge 1$ in order to deduce that the type $(k, \ell,m)$ is $(2, 3, 8)$, or $(2, 4, 5)$ or $(2, 4, 6)$. However, we note that for $(k,\ell,m) = (2, 3, 7)$, $(2, 3, 10)$ or $(2, 3, 12)$ the estimates on $\U(S)$ in the proof of \cite[(4.16)]{Shih_1991} do not contradict the inequality $\U(S) \ge 1$, leaving these cases open. In these cases, we refine the estimates as a part of establishing the following more general proposition.
Recall that $f(x)$ is the number of fixed points of $x \in G$ on $\Omega$.

\begin{proposition}\label{prop:t=1}
%[{Refinement of \cite[(4.16)]{Shih_1991}}]\label{modified 4.16}
%There exists a constant $c>0$ such that whenever we have a 
For every degree-$n$ type B system $(G,S,\Omega)$ of genus $g \le 1+n/460$ with $\#S=3$, we have $\soc(G)\cong L^2$ for some finite nonabelian simple group $L$.
\end{proposition}

\begin{proof}
\iffalse For the first part, we work through the cases in Proposition~\ref{prop: r=3 cases}. By Lemma~\ref{lem: U formula} and the fact that $\M(x_i)\le \frac{1}{10}$ and $\frac{f(x_i)}{n}\le \frac{1}{60}$ for all $i$, we get \[1-\U(x_1)-\U(x_2)=\U(x_3)\le \frac{1}{10}+\frac{9}{10m}-\frac{\phi(m)}{12m}.\]

    First consider $S$ of type $(2,4,m)$. Since $\U(x_1)+\U(x_2)\le \frac{19}{24}$, we get $\frac{13}{120}\le \frac{9}{10m}-\frac{\phi(m)}{12m}$. Using the bound $\phi(m)>m^{\frac{\log 2}{\log 3}}$ as before, this cannot hold for $m>8$. In fact, one may quickly check that it does not hold for $m\ge 7$.

        For $(k,\ell)=(2,5),(2,6),(3,3),(3,4)$ we use the same argument with the bounds $\U(x_1)+\U(x_2)\le \frac{433}{600}, \frac{263}{360}, \frac{31}{45}, \frac{113}{180}$.
    \fi
Since $G$ is of type B, we write $\soc(G)=L^{2t}$,  claiming that $t=1$. 
%We  show that $t=1$ which implies $\soc(G)\cong L^2$. 
By Proposition~\ref{prop: r=3 cases}, we may assume $S$ is one of the types:
\begin{equation}
\label{equ:list}
\begin{split}
  &   (2, 3,m)\text{ for }7\le m; 
 (2, 4,m), 5 \le m \le 37;
 (2, 5,m), 5 \le m \le 13;
    (2, 6,m), 6\leq m \leq  9; \\
& (3, 3,m), 4\le m \le 9;
(3, 4, m), m=4, 5.
\end{split}
\end{equation}
%1/360>2(g-1)/n=1-u(S)
Recall that the map $\rho:G \rightarrow S_t$ is defined by $L^x_s = L_{\rho(x)(s)}$.
First note that upon replacing the inequality $\U(S)\leq 1$ by $\U(S)\leq 359/360$ or equivalently $g<1+n/720$,  \cite[(4.18-20)]{Shih_1991} still apply and show\footnote{\cite[(4.20)]{Shih_1991} has a typo, namely, the number $307/1800$ should be replaced by $317/1800$, but this does not change the outcome.} that for $S$ of type $(2,3,8)$, $(2,4,5)$, and $(2,4,6)$, one has $t=1$.
%This proof works for \SPENCER{$c<\frac{1}{720}$}.
%4.18 works for $c<7/1200$, 4:19 works for $c< 1/720$, 4:20 works for $c<1/720$.
%\EILIDH{Typo in Shih (4.20): it should be 317/1800, not 307/1800. Doesn't change the result though.}
%\textbf{Case 1:} 
Moreover, if $\rho(x_i)=1$ for some $x_i \in S$, 
the argument of \cite[(4.17)]{Shih_1991} shows that the inequality $g\le 1+cn$ for $c<1/360$ implies $t=1$. %\SPENCER{Should this be $c=\frac{1}{460}$, as in statement of Proposition?}
%which works for \SPENCER{$c< \frac{1}{360}$} to show that $t=1$.
%{Shih 4.17 works for $c\le min{404999/16200000, 1079999/21600000, 35656847999999/377913600000000,79/4800,77/1500,1/40,1249/9000,47/16800,13/1200,43/10800,1/90,1/360,13/1800,35656847999999/377913600000000,2699/432000,13/1200}$}  

%\textbf{Case 2:} 
Henceforth assume that $S$ is as in 
\eqref{equ:list} but is not of type $(2,3,8)$, $(2,4,5)$, or $(2,4,6)$, and that  $\rho(x_i)\ne 1$ for all $x_i\in S$.
The last assumption implies that  $\text{fpr}(x_i)\le {1}/{60}$ for all $x_i \in S$ by \cite[(4.6)(3)]{Shih_1991}.

% Code to check constants in Shih 4.17 smallest c=1/360.
\begin{comment}
%k:=61/120;
%l:=7/25;
%m:=59/500;

%l:=1/4;
%m:=57/300;

%l:=103/400;
%m:=1/10+(1/7)*9/10;

%l:=57/300;
%m:=1/10+(1/5)*9/10;

%k:=1801/5400;
%l:=k;
%m:=1/10+(1/4)*9/10;

%//i=1 (255)
%k:=11/20;
%l:=1/5+(4/5)*1/60^4;
%m:=1/5+(4/5)*1/60^4;
%// (2510)
%m:=3/20+(2/5)*1/60^4;

%// l=4,6
for m in [9..15] do
y:=1/15;
1/2*(9/10*(1/m)-(1/12)*(EulerPhi(m)/m)-y);
end for;
"Now (26m)";
for m in [5..12] do
y:=23/180;
1/2*(9/10*(1/m)-(1/12)*(EulerPhi(m)/m)-y);
end for;

"Now (34m)";
for m in [3..12] do
y:=1/4;
1/2*(9/10*(1/m)-(1/12)*(EulerPhi(m)/m)-y);
end for;

//(248)
k:=11/20;;
l:=11/40+(1/2)*1/(60)^3;
m:=13/80+(1/2)*1/(60)^3;

k:=11/20;;
l:=103/400;
m:=41/240;
1/2*(1-k-l-m);
\end{comment}
% Check for (4.17)-(4.20)
\begin{comment}
    // (238)
k:=61/120;;
l:=1801/5400;
m:=11/80+31/(2*60^3);
m:=29/200;
l:=67/200;
m:=29/200;

// (245)
l:=17/60;
m:=1/5*(1+4/60^4);
m:=37/180;

// 246
k:=11/20;
l:=103/400;
m:=307/1800;

k:=61/120;
l:=17/60;
m:=37/180;
1/2*(1-k-l-m);

\end{comment}

%The details are listed in the table. 
%Most of the calculations are done in \cite[(4.16)]{Shih_1991}, but we correct some errors. 
By Riemann--Hurwitz $\U(S)=1+2(1-g)/n\geq 1-2/460=229/230$. The combination of this  with  Lemma~\ref{lem: U formula} and the bounds $\M(x_i)\le {1}/{10}$ \cite[(4.7.1)]{Shih_1991} and $\text{fpr}(x_i)\le {1}/{60}$, for all $i$, give: \begin{equation}\label{eqn: phi bound}\frac{229}{230}-\U(x_1)-\U(x_2)\le\U(x_3)\le \frac{1}{10}+\frac{9}{10m}-\frac{\phi(m)}{12m}.\end{equation}

Again using Lemma~\ref{lem: U formula} and the bounds $\M(x_i)\le {1}/{10}$ and $\text{fpr}(x_i)\le {1}/{60}$, for each of the possibilities for $(k,\ell)$ in \eqref{equ:list} we get the following upper bounds on $\U(x_1)+\U(x_2)$.

\vspace{2pt}
\begin{center} 

 \begin{tabular}{|c|c|}
        \hline
    $(k,\ell)$ & \text{upper bound on $\U(x_1)+\U(x_2)$} 
    \\
    \hline\\[-1em]
    $(2,3)$ & $\frac{307}{360}$ \\[5pt]
    $(2,4)$ & $\frac{19}{24}$ \\[5pt]
    $(2,5)$ & $\frac{433}{600}$ \\[5pt]
    $(2,6)$ & $\frac{13}{18}$ corrected from \cite{Shih_1991}\\[5pt]
    $(3,3)$ & $\frac{31}{45}$ \\[5pt]
    $(3,4)$ & $\frac{113}{180}$ \\[5pt]
    \hline
    \end{tabular}
\end{center} 
\vspace{2pt}

First we will bound the values of $m$ that may appear for $(k,\ell)=(2,3)$. By \cite[Section~4]{Shapiro1943}, unless $m=1, 2, 3, 4, 6, 10, 12, 18, 30$, we have $\phi(m)>m^{{\log 2}/{\log 3}}$. Thus for $m\ge 31$,
we get a contradiction for: 
\begin{equation}
    \label{eqn: phi bound2}
c<\frac{1}{49}<\frac{1}{2}\left(\frac{9}{10}-\frac{307}{360}-\frac{9}{310}+\frac{31^{\frac{\log 2}{\log 3}}}{372}\right).
\end{equation}

Now we have a finite list of types to check, and a direct computation of (\ref{eqn: phi bound}) with $m<31$ and $c<{47}/{5040}$ rules out all but the cases\footnote{The bounds in \cite[(4.16)]{Shih_1991} do not rule out these cases, leaving these cases open.} where $S$ has type $(2,3,7)$, $(2,3,10)$ or $(2,3,12)$.
    For these types, we apply better fixed-point ratio estimates.
    %Consider $(k,\ell,m) = (2, 3, 7)$. If $t\ge 2$ then since $2, 3, 7$ are pairwise coprime, there is no $i$ such that $\rho(x_i)=1$ by \cite[(4.13)]{Shih_1991}. Therefore 
    
    In the case $(2, 3, 7)$, since $\rho(x_3)\neq 1$ by assumption, $\rho(x_3)$ must contain a $7$-cycle and hence $\text{fpr}(x_3) \le {1}/{60^6}$ \cite[(4.6)(3)]{Shih_1991}. Thus Lemma~\ref{lem: U formula} 
    %Since \cite[(2.1)(1)]{Shih_1991} 
    gives $\U(x_3) \le \left(1 + 6/{60^6}\right)/7=\frac{1110857143}{7776000000}$. As $\U(x_1) + \U(x_2) \le {307}/{360}$, this gives $\U(S) \le\frac{7742057143}{7776000000}$, contradicting $\U(S) \ge 1-2c$ for {$c= 1/460< \frac{33942857}{15552000000}$}. %, for example $c\le\frac{1}{460}$. 
    This rules out $(2,3,7)$ when $t\ge 2$.

If $(k,\ell,m) = (2, 3, 10)$,
\begin{comment}
then \cite[(2.1)(1)]{Shih_1991} gives:

\begin{equation}\label{eqn: non-A_5}
    \U(x_3)= \frac{1}{10}\left( 1+\frac{f(x_3^5)}{n}+4\frac{f(x_3^2)}{n}+4\frac{f(x_3)}{n}\right).
\end{equation}

If there is some $i$ such that $\rho(x_i)=1$ then by \cite[(4.13)]{Shih_1991} there is exactly one $i$ such that $\rho(x_i)=1$, moreover $i=2$ and $t=3$. Since $|x_3|=10\notin \{2,6\}$, \cite[(4.14)]{Shih_1991} implies that $\frac{f(x_2)}{n}\le \frac{1}{400}$ and also that $\rho(x_1)=\rho(x_3)$ are $2$-cycles. So by \cite[(4.15)]{Shih_1991} we see that $\frac{f(x_3)}{n}\le\frac{f(x_3^2)}{n}\le \frac{1}{100}$. Since $\frac{f(x_3^5)}{n}\le \frac{1}{10}$ we have $\U(x_3)\le \frac{59}{500}$ by (\ref{eqn: non-A_5}). Now we may use $\U(x_1)=\frac{11}{20}$ and $\U(x_2)\le \frac{67}{200}$ which yields $\U(S)\le \frac{953}{1000}$, contradicting $\U(S) \ge 1-2c$ for \SPENCER{$c<\frac{47}{2000}$.}
\end{comment}
 then $\rho(x_3)$ has order $2$, $5$ or $10$ by assumption. Thus by \cite[(4.6)(3)]{Shih_1991}, we have $\text{fpr}(x_3) \le {1}/{60}$. If $5 \mid |\rho(x_3)|$, then $\rho(x_3^2)$ contains a $5$-cycle, otherwise $2 \mid |\rho(x_3)|$ so $\rho(x_3^5)$ contains a $2$-cycle. By \cite[(4.6)(3)]{Shih_1991} this gives, resp.,
\[\text{fpr}(x_3^2)\le \begin{cases}
    \frac{1}{60^4} &\mbox{ if }5 \mid |\rho(x_3)|\\
    \frac{1}{12} &\mbox{ otherwise,}
\end{cases}\qquad\text{and}\qquad \text{fpr}(x_3^5)\le \begin{cases}
    \frac{1}{60} &\mbox{ if }2 \mid |\rho(x_3)|\\
    \frac{1}{10} &\mbox{ otherwise.}
\end{cases}%\text{resp}.
\] By Lemma~\ref{lem: U formula} this gives:
\begin{equation*}\label{eqn: non-A_5}
    \U(x_3)= \frac{1}{10}\left( 1+\text{fpr}(x_3^5)+4\,\text{fpr}(x_3^2)+4\,\text{fpr}(x_3)\right) \le \frac{17}{120}.
\end{equation*} Since $\U(x_1) + \U(x_2) \le {307}/{360}$, one gets $\U(S) \le {179}/{180}$, contradicting $\U(S) \ge 1-2c$ for {$c<{1}/{360}$}, ruling out the existence of systems of type $(2,3,10)$. % is ruled out.

Now consider $S$ of type $(2,3,12)$. Then $\U(x_1) + \U(x_2) \le {307}/{360}$ and by \cite[(4.6)(3)]{Shih_1991} we have $\text{fpr}(x_3) \le {1}/{60}$. If $3 \mid |\rho(x_3)|$ then $\rho(x_3^2)$ and $\rho(x_3^4)$ each contain a cycle of length at least $3$, otherwise $2 \mid |\rho(x_3)|$ so $\rho(x_3^3)$ contains a cycle of length at least $2$. By \cite[(4.6)(3)]{Shih_1991}, this gives
\[\text{fpr}(x_3^2)\le \begin{cases}
    \frac{1}{60^2} &\mbox{ if }2 \mid |\rho(x_3)|\\
    \frac{1}{10} &\mbox{ otherwise,}\end{cases} \qquad 
    \text{fpr}(x_3^4)\le \begin{cases}
    \frac{1}{60^2} &\mbox{ if }2 \mid |\rho(x_3)|\\
    \frac{1}{12} &\mbox{ otherwise,}
\end{cases}\] \[\text{fpr}(x_3^3)\le \begin{cases}
    \frac{1}{60} &\mbox{ if }3 \mid |\rho(x_3)|\\
    \frac{1}{10} &\mbox{ otherwise.}
\end{cases}.\] Thus, \begin{equation*}\label{eqn: orbits order 12}
    \U(x_3)= \frac{1}{12}\left( 1+\text{fpr}(x_3^6)+2\,\text{fpr}(x_3^4)+2\,\text{fpr}(x_3^3)+2\,\text{fpr}(x_3^2)+4\,\text{fpr}(x_3)\right)\le \frac{47}{360}
\end{equation*} and one gets $\U(S) \le {59}/{6
0}$, contradicting $\U(S) \ge 1-2c$ for {$c<{1}/{120}$}.

\end{proof}

\begin{rem}\label{rem:r=4,t=2}
    We note that if $\#S=4$, then $t=1$ and $\soc(G)\cong A_5^2$, as well. Indeed, if $\#S=4$ then  $L\cong A_5$ and $\soc(G)\cong L^{2t}$ by Proposition~\ref{prop:br}. 
    Similarly to the $\#S=3$ case, 
    %and using the same map $\rho:G \rightarrow S_t$, % defined by $L^x_s = L_{\rho(x)(s)}$ as we did for $\#S=3$, 
    if $\rho(x_i)\neq 1$, 
    the proof of \cite[(4.9).(ii), L.\ 4-6]{Shih_1991} shows that $g\geq 1+n/40$, forcing $\rho$ to be the identity map, whence $t=1$.
\end{rem}

To narrow down to a finite list of groups for types $(2,3,7)$ we use:
\begin{lemma}[{Extension of \cite[Lemma~2.1(5)]{Aschbacher_1990}}]\label{lem: 237 centralizers}~
    Let $L$ be a finite nonabelian simple  group and $x\in \Aut(L)$ be of order $7$. Then either $|x^L|\ge 89$ or one of the following holds:
    \begin{itemize}
        \item $L\cong \PSL_2(7)$ and $|x^L|=24$;
        \item $L\cong \PSL_2(8)$ and $|x^L|=72$.
    \end{itemize}
\end{lemma}
\begin{proof}
    When $L=A_M$ for  $M\ge 9$, we have $|x^L|\ge 25920$ for all $x\in \operatorname{Aut}\left(L\right)$ by \cite[(3.4)]{Shih_1991}.
%    If $L\cong A_s$ then we claim there is no $x\in G$ of order $7$ such that $\M(x)=\frac{1}{85}$, and so $\M(x)<\frac{1}{85}$. Note that \[|x^L|={s \choose 7} \cdots {s-7(u-1) \choose 7}\frac{(6!)^u}{u!\epsilon}\] where $x$ has $u$ cycles of length $7$ and $\epsilon=1, 2$ (see \cite[Proof of 2.1]{Aschbacher_1990}). If $|x^L|=85$ then \[{s \choose 7} \cdots {s-7(u-1) \choose 7}(6!)^u=5\cdot 17 \cdot u!\epsilon\] but $3^{2u}$ divides the left hand side of this equation and not the right hand side.
Henceforth assume $L$ is a group of Lie type. Of course we only consider groups whose order is divisible by $7$. First we tackle the classical case.
    The bounds from Burness \cite{Burness_2007}, Corollary~3.38, Remark~3.13, Lemma~3.20 and Proposition~3.22, give us $|x^L|\ge 89$ for all but the following  groups:
    \begin{itemize}
        \item $\PSL_2(7)$, $\PSL_2(8)$, $\PSL_2(13)$. %\EILIDH{Note PSL(4,2) is $A_8$ so we already have it!}
        \item $\PSU_3(3)$, $\PSU_3(5)$;
        \item $\PSp_{4}(7)$, $\PSp_{6}(2)$, $\PSp_{8}(2)$ and $\PSp_{10}(2)$
        \item 
        $\PO_5(7)$, 
        $\PO_6^+(2)$, $\PO_6^+(4)$, $\PO_6^-(3)$, $\PO_7(2)$, $\PO_7(3)$, $\PO_8^\pm(2)$, $\PO_9(2)$, $\PO_{10}^\pm(2)$ and $\PO_{11}(2)$. 
 %       $P\Omega_5(7)$, $P\Omega_6^+(2)$, $P\Omega_6^+(4)$, $P\Omega_6^-(3)$, $P\Omega_7(2)$, $P\Omega_7(3)$, $P\Omega_8^\pm(2)$, $P\Omega_9(2)$, $P\Omega_{10}^\pm(2)$ and $P\Omega_{11}(2)$
    \end{itemize}
    
    Bounds of Burness and Thomas \cite[Table~4]{Burness2021} give us $|x^L|\ge 89$ for each exceptional group of Lie type.
    We now have a finite list of groups: the small groups of Lie type listed above and the 26 sporadic groups. A computation using GAP finishes the proof. See ``ConjugacyClassBounds.gap'' in \url{https://neftin.net.technion.ac.il/files/2025/01/b-code.zip} or the ancillary files of this paper at \url{https://arxiv.org/abs/2501.15538} for the GAP code.
    %\DN{I understood that there were exceptional groups of Lie type for which the bounds did not apply and Gap/Magma were used. Could those be listed?}\EILIDH{That was when I was using some weaker bounds. Then I found the Burness and Thomas paper, whose bounds rule out all exceptionals.}
\end{proof}

We now narrow down our search to a finite list of types $S$ and finitely many socles $L$.

\begin{proposition}\label{prop: socle} Assume a degree-$n$ group $G$ of type B admits a system $(G,S,\Omega)$ of genus $g\leq  1+n/5000$ with $\soc(G)\cong L^2$ and $\#S=3$. Then one of the following holds:
\begin{enumerate}
    \item $S$ is of type $(2,3,7)$ and $L \cong {\rm PSL}_2(7),{\rm PSL}_2(8), {\rm PSL}_2(13), A_7$ or $A_8$.
    \item $S$ is of type $(2,3,8)$ and $L \cong {\rm PSL}_2(7), {\rm PSL}_2(9), {\rm PSL}_2(16), {\rm PSL}_2(25), \rm{PSU}_4(2), A_6$ or $A_8$.
    \item %The remaining possible types of 
    $S$ is one of the types  listed in Proposition~\ref{prop: r=3 cases}, and %in these cases 
    $L \cong {\rm PSL}_2(7), A_5, A_6, A_7$ or $A_8$.
\end{enumerate}
  
\end{proposition}
\begin{proof}
If $(k,\ell,m)$ is not $(2, 3, 7)$ or $(2, 3, 8)$, then Shih's argument from  \cite[(4.24)(1)]{Shih_1991}, applies even merely when $g\leq 1+cn$ for $c<1/720$, giving $L \in \{{\rm PSL}_2(7), \,A_M \mid M \le 8\}$. 
%This completes the proof of the third part.
It therefore remains to consider types $(2,3,7)$ and $(2,3,8)$. 
% Code to check right k,ell,m tuples:
%l:=1/36;
%(1-3*l)/(1-l);
%(2/3+1/4)+1/144-33/35;
%1/2+2/5+1/144-33/35;
%for m in [7,8,9,10,11] do
%<m,1/2+1/3+1/m+1/144-33/35>;
%end for;
%for m in [5,6,7] do
%<m,1/2+1/4+1/m+1/144-33/35>;
%end for;

We claim that when  $g\leq 1+cn$ for $c\leq 1/5000<{12}/{52955}$ and the type  $(k,\ell,m)$ is $(2,3,7)$ or $(2, 3, 8)$, there must be some $x\in \text{Aut}(L)$ of order $2$ or $3$ with $|x^L|< 85$ or $x\in \text{Aut}(L)$ of order $7$ with $|x^L|< 89$.
%Now we have shown that for $c<\frac{12}{52955}$ if $(k,\ell,m)$ is $(2,3,7)$ or $(2, 3, 8)$, there must be some $x\in \text{Aut}(L)$ of order $2$ or $3$ with $|x^L|< 85$ or $x\in \text{Aut}(L)$ of order $7$ with $|x^L|< 89$.

If $(k,\ell,m)$ is $(2, 3, 8)$ and $g\le 1+cn$ for $c<{3/340}$, then the same argument as in \cite[(4.24)(2)]{Shih_1991} applies, showing there is some $x\in \text{Aut}(L)$ of order $2$ or $3$ with $|x^L|<85$.
For $(k,\ell,m)=(2, 3, 7)$ the argument for \cite[(4.24)(2)]{Shih_1991} needs to be sharpened as follows. If for all $x\in \text{Aut}(L)$ of order $2$ or $3$ we have $|x^L|\ge 85$ and for all $x\in \text{Aut}(L)$ of order $7$ we have $|x^L|\ge 89$, then  $\M(x_1), \M(x_2)<{1/85}$ and $\M(x_3)<{1/89}$ by \cite[(4.2)]{Shih_1991}. Now  \cite[(4.10)]{Shih_1991} gives the bounds \[\U(S)\le \frac{1}{85}+\frac{1}{85}+\frac{1}{89}+\frac{84}{85}\cdot \frac{1}{2}+\frac{84}{85}\cdot\frac{1}{3}+\frac{88}{89}\cdot\frac{1}{7}=\frac{52931}{52955}<1-2c\] for any $c<{12}/{52955}$. Therefore if  $G$ admits a system of genus $g\le 1+cn$ for $c<{12}/{52955}$, we must have some $x\in \text{Aut}(L)$ of order $2$ or $3$ with $|x^L|< 85$ or $x\in \text{Aut}(L)$ of order $7$ with $|x^L|< 89$.

Now we may apply Lemma~\ref{lem: 237 centralizers} and \cite[(3.5-6)]{Shih_1991} to get the list of all simple groups $L$ with $x\in \text{Aut}(L)$ of order $2$ or $3$ with $|x^L|< 85$ or $x\in \text{Aut}(L)$ of order $7$ with $|x^L|< 89$. These are ${\rm PSL}_2(q)$ for $q\le 16$, ${\rm PSL}_2(25)$, ${\rm PSU}_3(3)$, ${\rm PSU}_4(2)$, ${\rm PSp}_6(2)$, $A_5$, $A_6$, $A_7$ and $A_8$.

We will remove groups from this list until we have the claimed result:

First we remove groups $L$ with no element of order $7$ or $8$ in $\Aut(L)$. For type $(2,3,7)$ this leaves us with the list ${\rm PSL}_2(7)$, ${\rm PSL}_2(8)$, ${\rm PSL}_2(13)$, ${\rm PSU}_3(3)$, ${\rm PSp}_6(2)$, $A_7$ and $A_8$. For type $(2,3,8)$ the list is ${\rm PSL}_2(7)$, ${\rm PSL}_2(9)$, ${\rm PSL}_2(16)$, ${\rm PSL}_2(25)$, ${\rm PSU}_3(3)$, ${\rm PSU}_4(2)$, ${\rm PSp}_6(2)$, $A_6$ and $A_8$.

%Consider $L\cong{\rm PSU}_4(2)$. Since ${\rm Aut}\left({\rm PSU}_4(2)\right)$ has no elements of order $7$, and the only elements of order $8$ are outer automorphisms. . \EILIDH{But the element of order 2 with class size 36 is also outer, and the element of order 3 with class size 40 is inner... The only other outer automorphism of order 2, 3 or 8 is order 2 with class size 540. We may need to actually check all these possible tuples. Luckily the automorphism group is just twice the size of the group.}

It remains only to remove ${\rm PSU}_3(3)$ and ${\rm PSp}_6(2)$ from the list. We use the following upper bounds on $\M(x)$ for $x$ of order $2$, $3$, $7$ and $8$. For elements of order $2$ and $3$ these bounds come from \cite[(3.5)]{Shih_1991}, while for elements of order $7$ and $8$ they come from character tables found in GAP or MAGMA.
%21163/21420, 4079633/4112640

\begin{center}
\begin{tabular}{c|c|c|c|c}
     $L$& Order $2$ & Order $3$ & Order $7$ & Order $8$ \\
     \hline\\[-1em]
     ${\rm PSU}_3(3)$& $\frac{1}{63}$& $\frac{1}{56}$ & $\frac{1}{864}$ & $\frac{1}{63}$\\[5pt]
       ${\rm PSp}_6(2)$& $\frac{1}{63}$ & $\frac{1}{85}$& $\frac{1}{207360}$ & $\frac{1}{63}$\\
\end{tabular}
\end{center}

Now we apply the bounds from \cite[(4.10)]{Shih_1991}. In the case $L\cong{\rm PSU}_3(3)$, if $S$ of type $(2, 3, 7)$ the bound is \[\U(S)\le \frac{1}{63}+\frac{1}{56}+\frac{1}{864}+\frac{1}{2}\cdot\frac{62}{63}+\frac{1}{3}\cdot \frac{55}{56}+\frac{1}{7}\cdot \frac{863}{864}=\frac{335}{336}<1-2c\] for $c<{1}/{672}$, while if $S$ is of type $(2, 3, 8)$ the bound it gives is: 
\[\U(S)\le \frac{1}{63}+\frac{1}{56}+\frac{1}{63}+\frac{1}{2}\cdot \frac{62}{63}+\frac{1}{3}\cdot\frac{55}{56}+\frac{1}{8}\cdot\frac{62}{63}=\frac{125}{126}<1-2c,\]
for $c<{1}/{252}$. 
Similarly for $L\cong {\rm PSU}_4(2)$ we get $\U(S)<1-2c$ for any $c$ which is at most $1/5000<{25}/{2592}$, and for $L\cong {\rm PSp}_6(2)$ we get $\U(S)<1-2c$ for any $c$ which is at most $1/5000<{33007}/{8225280}$, completing the proof.

\end{proof}

%The following code checks the constants for shih's pf of the above prop:
%\begin{comment}
%    (1-30061/30240)-1/200;
%1-(1/2+1/3+1/7+4/195)-1/305;
%(1-(1/2+1/3+1/7+347/16128))-1/440;
%1-(1/2+1/3+1/8+101/2520);
%1-(1/2+1/3+1/8+17/504);
%1-(1/2+1/3+1/8+104/2835)-1/630;
%3/85+(84/85)*(1/2+1/3+1/7);
%3/85+(84/85)*(1/2+1/3+1/8);
%\end{comment}
Finally, we check the remaining tuples using MAGMA:
\begin{proposition}\label{prop:Spencer}
Assume $(G,S,\Omega)$ is a genus-$g$ system for $G$ of one of the types mentioned in Propositions \ref{prop:br}, \ref{prop: r=3 cases} and \ref{prop: socle}, so that $G$ is a group of type B with socle $L^2$ for $L=A_5,A_6,A_7,A_8$, or $\PSL_2(q)$, $q\in\{7,8,9,13,16,25\}$.  If  $g\leq 1$, then  
%    Assume a group $G$ of type $B$ admits a genus $g\leq  1+cn$ system $(G,S,\Omega)$. Then $g=1$, 
$G\cong {\rm PSL}_2(7)^2.C_2$; $S=(2,3,8)$; $|\Omega|=168$; and $g=1$. Moreover, for $g<6$ there are no genus-$g$ systems for $L=A_8$ and for  $L=\PSU_4(2)$.
\end{proposition}
\begin{proof}
For type-B groups $G$ with socle $L^2$ where $L=A_5, A_6,A_7, {\rm PSL}_2(q)$, \\$q\in\{7,8,9,13,16,25\}$, we verify below directly using Magma \cite{Magma} that there is only one system $(G,S,\Omega)$ of genus $g\leq 1$. See ``SmallL.magma'' in \url{https://neftin.net.technion.ac.il/files/2025/01/b-code.zip} or the ancillary files of this paper at \url{https://arxiv.org/abs/2501.15538} for the relevant Magma code, and ``Example.magma'' for a sample computation. 

To start, consider the cases $L=A_5, A_6,A_7$, and $g\leq 1$. For all elements $x_1, x_2,x_3$ of orders $k,\ell,m$ in $G$, where $(k,\ell,m)$ is a type listed in Proposition \ref{prop: r=3 cases}, the Magma program determines whether $x_1, x_2$ and $x_3$ have product one, satisfy the genus-$g$ Riemann-Hurwitz condition and generate $G$. See ``SmallL.magma'' for the exact list of cases checked. For $L=A_5$  and $g\leq 1$, the same property is determined for all elements $x_1, x_2,x_3, x_4$ of orders $2,2,2,3$ in $G$. 

Now assume $L={\rm PSL}_2(q)$ and $g\leq1$. For $q \in \{7,8,13\}$, the Magma program determines whether there are any elements of orders $2,3,7$ that form a genus $g$-system. For $q \in \{7,9,16,25\}$, the same property is determined for all elements of orders $2,3$ and $8$. The computations reveal that in the cases under consideration, a genus $g\leq 1$ occurs only when $g=1$, $S=(2,3,8)$,  $G\cong {\rm PSL}_2(7)^2.C_2$, and  $|\Omega|=168$. 

To complete the proposition, we must consider type-B groups with $L=A_8$ or $L=\PSU_4(2)$. Up to isomorphism, there are two type-B groups $G$ with $L=A_8$. These are $A_8^2$, and the extension $A_8^2.C_2$ of $C_2$ acting diagonally as conjugation by a transposition in $S_8$. Given the size of these groups, a slightly different approach is needed to show that there is no genus $g<6$ system $(G,S,\Omega)$. 

 Assume $S=(k,\ell,m)$, and $L=A_8$ with $n=|L|$. To show there is no genus $g<6$ system $(G,S,\Omega)$, clearly it suffices to show that there are no elements $x_1,x_2,x_3$ of orders $k,\ell$ and $m$ in $G$ satisfying the genus-$g$ Riemann--Hurwitz condition $\U(S)=2(1-g)/n$. Note that $\U(x)$ is independent of the choice of representative in $x^G$. Hence to check this property, it suffices to check whether the genus-$g$ Riemann--Hurwitz condition
holds for representatives of each conjugacy class of elements of orders $k,\ell$ and $m$ in $G$. Magma determines this for all conjugacy classes of elements of orders $k,\ell,m $ in $G$, where $(k,\ell,m)$ is listed in Proposition \ref{prop: r=3 cases}.
See ``LargeL.magma'' in the above url for the code, and the exact list of cases that are checked.

For $L=\PSU_4(2)$ we have $S=(2,3,8)$ by Proposition~\ref{prop: socle} and we note that elements of order $8$ only appear as outer automorphisms in $\Aut(L)$, therefore we must only consider $G\cong {\rm PSU}_4(2)^2.C_2$, which we treat in the same way as above.

 The computation reveals that the genus-$g$ Riemann--Hurwitz condition is satisfied only when $g=1$, $G$ is isomorphic to $A_8^2.C_2$, and $S=(2,3,7)$.  To complete the proof, we must rule out this final case. This is straightforward to do. If $x_1,x_2,x_3$ are elements of order $2,3,$ and 7 that generate $A_8^2.C_2$, then $x_i\not\in A_8^2$ for at least one $i$. However, since $x_1x_2x_3=1$, we must have two such elements, but these elements have even order, contradicting the fact that the tuple is of type $(2,3,7)$.

\end{proof}
\begin{proof}[Proof of Theorem \ref{main}]
As in RET, a degree-$n$ covering $f:X\to\mathbb P^1_\mC$ of genus $g<\max\{2,n/5000\}$ whose monodromy group $G$ is of type B defines a genus-$g$ system $(G,S,\Omega)$.  Suppose the minimal normal subgroup of $G$ is a power of the simple group $L$. By the combination of Propositions~\ref{prop:br}, \ref{prop: r=3 cases}, and \ref{prop: socle}, we reduce to cases (1)--(3) in Proposition \ref{prop: socle}, where $\soc(G)=L^2$ %(\DN{what about $2,2,2,3$ is it with $t=1$?}) 
and $L$ is one of $\PSL_2(q), q=7,8,9,13,16,25$; $A_k$, $5\leq k\leq 8$, $\PSU_4(2)$ and to the exceptional case in Proposition \ref{prop:br} and Remark \ref{rem:r=4,t=2}. Since $n=\#L$ as recalled in \S\ref{sec:prelim}, among these groups we have $n/5000> 1$ only for $L=A_8$ and $\PSU_4(2)$. For these groups
%$A_8$ (resp.\ $\PSU_4(2))$ 
one has $g\leq \lfloor n/5000\rfloor < 6$, %(resp.\ $\PSU_4(2)$), 
and hence Proposition \ref{prop:Spencer} yields that there are no such tuples. For the other groups listed in Proposition \ref{prop: socle} one has $n/5000\leq 1$, and hence $g\leq 1$. Proposition \ref{prop:Spencer} then implies $(G,S,\Omega)$ is a genus-$1$ system of degree $n=168$ of type $(2,3,8)$ with $\soc(G)=\PSL_2(7)^2$. %already found in \cite{Salih2023}.
%$S$ is of type $(2,3,8)$
%    and \ref{prop:Spencer}.
\end{proof}
\begin{rem}\label{ram}
    As already found by Salih  \cite[Table 5]{Salih2023}, there are in fact two ramification types associated to the resulting type-(2,3,8) genus-$1$ systems of the Theorem. The conjugacy class of elements of order $3$ is the unique conjugacy class of order $3$ elements in $\PSL_2(7)^2$. The conjugacy classes of elements of order $2$ is the unique conjugacy class of order $2$ elements in $G$ that is not contained in $\PSL_2(7)^2$. The code in Construction.magma constructs the two conjugacy classes of order-$8$ elements in $G$ which are involved in genus-$1$ tuples for $G$, and hence there is a total of two associated ramification types.   %The conjugacy class of order $8$ elements can be either of the two conjugacy classes of order $8$ elements in $G$.     
\end{rem}

\bibliography{bib}
\bibliographystyle{plain}

\end{document}